\documentclass[letterpaper]{amsart}

\usepackage{amsmath}
\usepackage[utf8]{inputenc}
\usepackage{lmodern}
\usepackage{microtype}
\usepackage{geometry}
\usepackage{mathtools}
\usepackage{amsthm}
\usepackage{thmtools}
\usepackage{thm-restate}
\usepackage{amsfonts}
\usepackage{amssymb}
\usepackage[titletoc,toc,title]{appendix}
\usepackage{tikz}
\usepackage{xcolor}
\usepackage{todonotes}
\usepackage{bm}
\usepackage{etoolbox}
\usepackage{graphicx}
\usepackage{enumitem}

\setlist[itemize]{leftmargin=15mm}

\graphicspath{ {./images/} }

\usepackage{hyperref}
\usepackage{enumitem}
\usepackage{enumitem}
\setlist{nolistsep}

\declaretheoremstyle[
headfont=\normalfont\bfseries,
parent=subsection,
bodyfont=\normalfont,
qed=$\dashv$,
]{def2}

\declaretheorem[style = def2, title = Definition]{defn}

\theoremstyle{definition}

\newtheorem{example}[defn]{Example}

\newtheorem{openq*}{Open Question}
\newtheorem{openq}[defn]{Open Question}

\theoremstyle{remark}

\newtheorem{no-rem}[defn]{Remark}

\theoremstyle{plain}

\newtheorem{prop}[defn]{Proposition}
\newtheorem{lem}[defn]{Lemma}

\newtheorem{claim}[defn]{Claim}

\newtheorem{manualtheoreminner}{Theorem}
\newenvironment{manualthm}[1]{%
  \IfBlankTF{#1}
    {\renewcommand{\themanualtheoreminner}{\unskip}}
    {\renewcommand\themanualtheoreminner{#1}}%
  \manualtheoreminner
}{\endmanualtheoreminner}

\newcounter{cases}
\newcounter{subcases}[cases]

\makeatletter 
\def\l@subsection{\@tocline{2}{0pt}{1pc}{5pc}{}} \def\l@subsection{\@tocline{2}{0pt}{2pc}{6pc}{}} \makeatother

\newcommand{\ordb}{\textup{ord}(b)}
\newcommand{\ordbb}{\textup{ord(}{\ensuremath{\bar{b}}})}
\newcommand{\ordg}{\textup{ord(}{\ensuremath{\bar{g}}})}

\title{On the Effectiveness of Partition Regularity over Algebraic Structures}
\author{Gabriela Laboska}
\thanks{This work is part of the author's doctoral thesis at the University of Chicago, jointly advised by Denis Hirschfeldt and Maryanthe Malliaris.}
\address{Department of Mathematics\\
	University of Chicago\\
	5734 S University Ave\\
	Chicago, IL 60637}
\email[Gabriela Laboska]{gabrielap@uchicago.edu}

\begin{document}
\maketitle

\begin{abstract}

Partition regularity over algebraic structures is a topic in Ramsey theory that has been extensively researched by combinatorialists (\cite{cite2}, \cite{cite3}, \cite{cite5}, \cite{cite15}). 
Motivated by recent work in this area, we investigate the computability-theoretic and reverse-mathematical aspects of partition regularity over algebraic structures—an area that, to the best of our knowledge, has not been explored before. This paper focuses on a 1975 theorem by Straus \cite{cite24}, which has played a significant role in many of the results in this field.
\end{abstract}

\tableofcontents

\section{Introduction}

\vspace{5mm}

The computability-theoretic and reverse-mathematical analysis of combinatorial principles has a long history and has been especially fruitful (see,
for example,  \cite{cite12}). In this paper, we analyze results in a subfield of Ramsey
theory called partition regularity over algebraic structures. This topic has been extensively researched by combinatorialists,
with new important results in the last few years (\cite{cite5}, \cite{cite15}). However, the computability-theoretic and
reverse-mathematical analysis of partition regularity over algebraic structures is something that, to
our knowledge, is done for the first time.

Partition regularity, in the context that we are going to be talking
about in this paper, was the main topic in Richard Rado's doctoral
thesis. Rado's goal was to generalize some of the earliest results
in Ramsey theory, like Schur's theorem and Van der Waerden's theorem.
These theorems, in their original form or restated, talk about linear
equations or systems of equations that have
monochromatic solutions under \textit{any} finite coloring of the underlying algebraic structure that we are
considering.
To make this more precise, let us consider the set of positive integers
$\mathbb{N}\setminus\{0\}$ as our underlying structure. If we are
given a system of linear equations over $\mathbb{N}\setminus\{0\}$
in several variables, and color $\mathbb{N}\setminus\{0\}$ with
finitely many colors, can we always find a nonzero monochromatic solution? Schur's Theorem answers this question in the
positive for the equation $x+y-z=0$.

The systems of linear equations that have a monochromatic solution under any finite coloring of the underlying structure
are called partition regular systems. We are going to give the precise
definitions in Section 2.1. Rado met his goal -- he found a characterization
of all systems that are partition regular, separately for the homogeneous
and inhomogeneous case. For the inhomogeneous case, which is the case
we are going to be investigating in this paper, Rado showed the following.

\begin{manualthm}{}[Rado]

An inhomogeneous system of linear equations is partition regular over $\mathbb{Z}$ if and only if it has a constant solution.

\end{manualthm}

Some natural questions were raised after Rado showed this result.
One of these questions was what happens if we replace the integers
with any other commutative ring $R$? Can we still characterize the
systems that are partition regular over $R$? Significant progress towards answering this question was made by several combinatorialists, see \cite{cite3} and \cite{cite5}, and nearly a hundred years later,
it was finally resolved by Leader and Russell in \cite{cite15}. 

\begin{manualthm}{}[Leader-Russell]
An inhomogeneous system of linear equations is partition regular over a commutative ring $R$ if
and only if it has a constant solution. 
\end{manualthm}

Their result is interesting not only because they solved this open problem, but because of their
approach in solving it, which was different from the previous work on the subject. The authors
that worked on versions of the problem previously, including Rado himself, used an indirect approach in similar results -- they first
showed that the result holds for one equation and then showed how to pass from one equation to a
system of equations. Leader and Russell’s approach was direct and with that, they gave a new proof
to Rado’s original theorem. Their result and its proof were the main
motivation behind this research project.

While exploring these new results, as well as earlier papers attempting to address the same question, we came across a result by Straus from a 1975 paper \cite{cite24}. Whether in
a direct application or as an idea to follow, this theorem, that we
are going to refer to as Straus' theorem, shows up in nearly all of
these papers, so we felt that it was essential to start with its analysis.
\begin{defn}
Let $G$ be any nonempty set, $k\in\mathbb{N},k>0$ and $c:G\rightarrow k$ be any coloring
of $G$ with $k$ colors. We say that the tuple $\left(x_{1},y_{1},x_{2},y_{2},\dots,x_{n},y_{n}\right)\in G^{2n}$
is\textit{ pairwise monochromatic for $c$} if for all $i=1,\dots,n$,
$c(x_{i})=c(y_{i})$.
\end{defn}

\begin{manualthm}{}[Straus]

Let $n\in\mathbb{N},n\geq1$. Let $G$ be an abelian group and $b\in G,b\neq0$ an element
of $G$. Then there is a finite $k$-coloring of $G$ such that the
equation

\begin{equation}
(x_{1}-y_{1})+(x_{2}-y_{2})+\dots+(x_{n}-y_{n})=b
\end{equation}

\noindent has no pairwise monochromatic solutions, where

\[
k=\begin{cases}
\hspace{4mm}2n & \text{if }2\mid \ordb \text{ or } \ordb=\infty\\
\left\lceil \dfrac{2np}{p-1}\right\rceil  & \text{if }\ordb \text{ is odd and }p\text{ is the largest prime divisor of }\ordb
\end{cases}.
\]

\end{manualthm}

For simplicity, we refer to this
result as Straus' theorem; however, the following more general result is shown in Straus' paper \cite{cite24}.

\begin{manualthm}{}[Straus*]

Let $n\in\mathbb{N},n\geq1$. Let $(G,+,0)$ be an abelian group and $b\in G,b\neq0$ an element of $G$. Let $f_1,\dots,f_n$ be arbitrary mappings from $G$ to $G$, with $m\leq n$ of them being distinct.
Then there is a finite $k$-coloring $c$ of $G$ such that the equation

\begin{equation}
(f_1(x_{1})-f_1(y_{1}))+(f_2(x_{2})-f_2(y_{2}))+\dots+(f_n(x_{n})-f_n(y_{n}))=b
\end{equation}

\vspace{3mm}
\noindent has no pairwise monochromatic solutions, where
\[
k=\begin{cases}
\hspace{4mm}(2n)^m & \text{if }2\mid \ordb \text{ or } \ordb=\infty\\
\left\lceil \dfrac{2np}{p-1}\right\rceil^m  & \text{if }\ordb \text{ is odd and }p\text{ is the largest prime divisor of }\ordb
\end{cases}.
\]
\end{manualthm}

Our results apply to the more general theorem as well and they will be denoted by *.

In this paper we analyze this result and its implications from a computability-theoretic (Section 3) and reverse-mathematical (Section 4) point of view. 

From the computability-theoretic point of view, we show that this result does not hold computably, even in the simplest
case when $n=1$. We also show that the Turing degrees that capture
the complexity of this problem are the PA degrees.  The PA degrees compute paths in infinite binary (or $k$-branching for a fixed $k\in\mathbb{N},k\geq 2$) trees. They come up naturally in these types of problems since
you can think of the colorings of an algebraic structure as paths
through some binary (or $k$-branching for a fixed $k\in\mathbb{N},k\geq 2$) tree (see Section 2.2). 

\begin{manualthm}{3.1}
There is a computable abelian group $(G,+,0)$ and an element $b\in G,b\neq0$
such that if $c$ is a 2-coloring of $G$ for which the equation 
\[
x-y=b
\]
 has no monochromatic solutions, then $c$ has PA degree.
\end{manualthm}

Theorem 3.1 gives us a lower bound for this problem and it is not too hard to see that the PA degrees are the best possible bound (Proposition 3.0.7).

This theorem is the best possible result of this sort: if we have
more than $2$ colors, say $m>2$ colors, then for every computable
group we can find a computable $m$-coloring with no monochromatic
solutions to $x-y=b$ (Proposition 3.0.8).

The proof of Straus' theorem, which we will sketch at the beginning of the proof of Theorem 4.2,  shows that you can choose your coloring such that there exists an $m>0$ such
that for every $k\in\mathbb{Z}$, all the $kmb$ are monochromatic. We show that for the case $n\geq2$, Straus' theorem with this additional condition added does not hold computably, and moreover, the best possible bounds for this problem are again the PA degrees. 

\begin{manualthm}{3.2}
Let $n\in\mathbb{N},n\geq2$. There is a computable group $(G,+,0)$ and an element $b\in G,b\neq0$ such that, if $c$ is a $2n$-coloring of $G$ for which the equation $$(x_{1}-y_{1})+...+(x_{n}-y_{n})=b$$ has no pairwise monochromatic solutions and there is an $m>0$ such that for every $k \in \mathbb{Z}$, all the $kmb$ are colored with the same color, then $c$ has PA degree.
\end{manualthm}

In Theorem 3.3 we show that a more general version of Theorem 3.2 holds -- we can build a computable group $G$ and a $b\in G,b\neq 0$ such that \textit{any} finite coloring that satisfies the conditions of Theorem 3.2 has PA degree.

From the reverse-mathematical point of view, we can already make some
conclusions based on these results. The PA degrees suggest that Straus'
theorem might be equivalent to $\textup{WKL}_{0}$ over the base system
$\textup{RCA}_{0}$. Indeed, from the proof of Theorem 3.1, we
notice that the
$n=1$ case of Straus' theorem implies $\textup{WKL}_{0}$, so we have the following result.

\begin{manualthm}{4.1}
The following statements are equivalent over $\textup{RCA}_0\textup{:}$
\begin{itemize}
\item[$\textup{(i)}$] $\textup{WKL}_0$
\item[$\textup{(ii)}$]  For every abelian group $G$ and every element $b\in G,b\neq0$,
there is a $k$-coloring of $G$ for which the equation $x-y=b$ has
no monochromatic solutions, where
\[
k=\begin{cases}
\hspace{4mm}2 & \text{if }2\mid \ordb \text{ or } \ordb=\infty\\
\hspace{4mm}3   & \text{if }\ordb \text{ is odd }
\end{cases}.
\]
\end{itemize}
\end{manualthm}

Following the original proof of Straus' theorem, for the case $n \geq 2$ we show that $\textup{WKL}_{0}$ implies Straus' theorem over $\textup{RCA}_{0}$.

\begin{manualthm}{4.2}
Over $\textup{RCA}_0$, $\textup{WKL}_{0}$ implies the following statement$\textup{:}$

Let $n\in\mathbb{N},n\geq1$. For every abelian group $(G,+,0)$ and every $b\in G,b\neq0,$ there
is a $k$-coloring of $G$ such that the equation
\begin{equation*}
(x_{1}-y_{1})+\dots+(x_{n}-y_{n})=b
\end{equation*}

\noindent has no pairwise monochromatic solutions, where 
\[
k=\begin{cases}
\hspace{4mm}2n & \text{if }2\mid \ordb \text{ or } \ordb=\infty\\
\left\lceil \dfrac{2np}{p-1}\right\rceil  & \text{if }\ordb \text{ is odd and }p\text{ is the largest prime divisor of }\ordb
\end{cases}.
\]
\end{manualthm}

This result together with Theorem 3.2 suggest that the full Straus’ theorem is indeed
equivalent to $\textup{WKL}_0$ over $\textup{RCA}_0$.

Lastly, in section 4.3, we give a reverse-mathematical version of Theorem 3.3 and Theorem 3.4. The statements $\textup{(iii)}_{N, N\geq k}$ in Theorem 3.3 form an infinite schema of statements for every $N\in\mathbb{N}, N\geq k$, where $k$ is as stated below.

\begin{manualthm}{4.3}
The following statements are equivalent over $\textup{RCA}_0\textup{:}$
\begin{itemize}
\item[$\textup{(i)}$] $\textup{WKL}_0$
\item[$\textup{(ii)}$] Let $n\in \mathbb{N},n\geq2$. Let $(G,+,0)$ be an abelian group and $b\in G, b\neq 0$. Then there is a $k$-coloring of $G$ with no pairwise monochromatic solutions to 
\begin{equation*}
(x_{1}-y_{1})+\dots+(x_{n}-y_{n})=b
\end{equation*}
that satisfies the following condition: there is an $m>0$ such that for every $l\in \mathbb{Z}$, all the $lmb$ are colored with the same color, where
\[
k=\begin{cases}
\hspace{4mm}2n & \text{if }2\mid \ordb \text{ or } \ordb=\infty\\
\left\lceil \dfrac{2np}{p-1}\right\rceil  & \text{if }\ordb \text{ is odd and }p\text{ is the largest prime divisor of }\ordb
\end{cases}.
\]

\item[$\textup{(iii)}_{N \geq k}$]  Let $n\in \mathbb{N},n\geq2$. Let $(G,+,0)$ be an abelian group and $b\in G, b\neq 0$. Then there is an $N$-coloring of $G$ with no pairwise monochromatic solutions to 
\begin{equation*}
(x_{1}-y_{1})+\dots+(x_{n}-y_{n})=b
\end{equation*}
that satisfies the following condition: there is an $m>0$ such that for every $l\in \mathbb{Z}$, all the $lmb$ are colored with the same color, where 
\[
k=\begin{cases}
\hspace{4mm}2n & \text{if }2\mid \ordb \text{ or } \ordb=\infty\\
\left\lceil \dfrac{2np}{p-1}\right\rceil  & \text{if }\ordb \text{ is odd and }p\text{ is the largest prime divisor of }\ordb
\end{cases}.
\]
\end{itemize}

\end{manualthm}

\begin{manualthm}{4.4}
The following statements are equivalent over $\textup{RCA}_0\textup{:}$
\begin{itemize}
\item[$\textup{(i)}$] $\textup{WKL}_0$
\item[$\textup{(ii)}$]  Let $n\in \mathbb{N},n\geq2$. Let $(G,+,0)$ be an abelian group and $b\in G, b\neq 0$. Then there is a finite coloring of $G$ with no pairwise monochromatic solutions to 
\begin{equation*}
(x_{1}-y_{1})+\dots+(x_{n}-y_{n})=b
\end{equation*}
that satisfies the following condition: there is an $m>0$ such that for every $l\in \mathbb{Z}$, all the $lmb$ are colored with the same color.
\end{itemize}
\end{manualthm}

In Section 5, we give some potential directions for this research and some open questions that came up while writing this paper. The most pressing ones are the following.

\begin{openq*}
Are the PA degrees the best possible computability-theoretic bound for the full Straus' theorem? If not, what are the best possible bounds?
\end{openq*}

On the reverse mathematics side, the following question is still open, even though we show one direction of the equivalence in Theorem 4.2.

\begin{openq*}
Is the full Straus' theorem equivalent to $\textup{WKL}_0$ over $\textup{RCA}_0$?
\end{openq*}

Our proofs do not directly translate to proofs about rings, so the same computability-theoretic and reverse-mathematical analysis does not directly apply to the Leader-Russell theorem. Therefore, we have the following open questions.

\begin{openq*}
What are the best possible computability-theoretic bounds for the Leader-Russell theorem? What about the reverse-mathematical bounds?
\end{openq*}

We can also consider some more refined reducibilities like Weihruach reducibility or different base systems of axioms like $\textup{RCA}_0^*$ and ask whether our theorems hold under these conditions. These questions are interesting to ask since the theorem by Jockusch that we use in most of our results does not necessarily hold under these new conditions.

\newpage

\section{Background}

\vspace{5mm}

In this section, we give the necessary background for the new results in this paper. We aim for the material to be self-contained, but we also give references throughout.

\subsection{Partition Regularity}
Partition regularity is a subfield of Ramsey theory that generalizes some important combinatorial results, such as Schur's Theorem and Van der Waerden's Theorem. For further details on these topics, see \cite{cite9}.

\begin{defn}
Let $G$ be a nonempty set and $k \in \mathbb{N}, k>1$. A \textit{$k$-coloring} of $G$ is any map $c : G \rightarrow k$. We call the set $S \subset G$ \textit{monochromatic} if $c(x)=c(y)$ for all $x,y\in S$.
\end{defn}

Since we will mostly be talking about equations and system of equations and their monochromatic solutions, we define precisely what we mean by this. 

\begin{defn}
Let $G$ be an algebraic structure (abelian group, ring), $A$ an $m\times n$ matrix with entries in $G$ and $b\in\mathbb{G}^{m},b\neq0$. We call the $n$-tuple $(x_1,x_2,\dots,x_n)$ a \textit{monochromatic solution} to the  system of linear equations $Ax=b$ if $A\cdot[x_1,x_2,\dots,x_n]^{\textup{T}}=b$ and the set $\{x_1,x_2,\dots,x_n\}$ is monochromatic.
\end{defn}

\begin{manualthm}{}[Schur]
For every finite coloring of the positive integers, there
is a nonzero monochromatic solution to the equation $x+y-z=0$.
\end{manualthm}
The original statement of Van der Waerden's theorem is the following:

\begin{manualthm}{}[Van der Waerden]
 For every finite coloring of the positive integers,
there are arbitrarily long monochromatic arithmetic progressions.
\end{manualthm}

With some care, this theorem can be restated as a statement about
monochromatic solutions to a well-chosen system of equations. Rado's
goal in his thesis was to generalize these results.
\begin{defn}
Let $A$ be an $m\times n$ matrix with entries in $\mathbb{N}$
and $b\in\mathbb{N}^{m}$. We say that the system
of linear equations $Ax=b$ is \textit{partition regular over $\mathbb{N}\setminus\{0\}$}
if for every finite coloring of the positive integers, the system
has a nonzero monochromatic solution.
\end{defn}
Using this notion, we can restate both Schur's and Van der Waerden's
theorem into statements about a specific system of equations being
partition regular. For example, the restated Schur's theorem is the
following:

\begin{manualthm}{}[Schur, restated]
 Let $A=\left[\begin{array}{ccc}
1 & 1 & -1\end{array}\right]$ and $b=0.$ Then the system $Ax=b$ is partition regular over $\mathbb{N}\setminus\{0\}$.
\end{manualthm}
Rado successfully characterized all partition regular systems over
$\mathbb{N}\setminus\{0\}$,\textit{ }dealing separately with the
cases $b=0$ and $b\neq0$. The case that we are going to be investigating
is the case when $b\neq0$, which is sometimes referred to as inhomogeneous
partition regularity. \textit{From now on, unless we specify which case we
are referring to, partition regularity for us will mean inhomogeneous
partition regularity.} We refer the reader to \cite{cite17} for more information
on the $b=0$ case. Note that, unlike the case $b=0$, for the nonzero
case, we do not have to worry about excluding zero solutions, so we
are going to use the following definition of partition regularity
over $\mathbb{Z}$.
\begin{defn}
Let $A$ be an $m\times n$ matrix with entries in $\mathbb{Z}$ and
$b\in\mathbb{Z}^{m},b\neq0$. We say that the system of linear equations
$Ax=b$ is \textit{partition regular over $\mathbb{Z}$} if for every
finite coloring of $\mathbb{Z}$, the system has a monochromatic solution.
\end{defn}
\begin{example}
Here are some examples and non-examples of partition regularity. 
\begin{enumerate}
\item The system 
$$
\begin{aligned}
2x + 3y - z + 4w &= 24 \\
x - 4y + 2z + w &= 0 \\
3x + y - 5z + 2w &= 3 \\
x + 2y + z - 3w &= 3
\end{aligned}$$
 is partition regular over $\mathbb{Z}$ for
the simple reason that it has a constant solution $(3,3,3,3)$, so no
matter how we color $\mathbb{Z}$, it will always have a monochromatic
solution. Clearly, any other system of equations that has a constant solution will be partition regular.
\item The equation $x+y=3$ is not partition regular over $\mathbb{Z}$.
Consider, for example, the coloring where all even integers are colored
red and all odd integers are colored blue. Since the solutions of
$x+y=3$ are always pairs of an even and an odd integer, there will
be no monochromatic solution for this coloring.
\end{enumerate}
\end{example}

These examples are relatively simple, but it can get much more complicated
to show partition regularity/non-regularity for different systems
of equations. One thing to note in this example is that, even though
we did not use this explicitly, the second equation does not have
a constant solution. It turns out that this condition is necessary
and sufficient for partition regularity over $\mathbb{Z}$.

\begin{manualthm}{}[Rado]
 Let $A$ be an $m\times n$ matrix with entries in $\mathbb{Z}$
and $b\in\mathbb{Z}^{m},b\neq0$. The system of linear equations $Ax=b$
is partition regular over $\mathbb{Z}$ if and only if it has a constant
solution.
\end{manualthm}
A few natural questions arise from this result. If we replace $\mathbb{Z}$
with any other commutative ring, will the result still hold? If not,
can we somehow characterize the rings for which the result does hold?
We first need to define partition regularity over a ring. The above definition generalizes pretty naturally.
\begin{defn}
Let $R$ be a commutative ring, $A$ an $m\times n$ matrix with entries
in $R$ and $b\in R^{m}$, $b\neq0$. We say that the system of linear
equations $Ax=b$ is \textit{partition regular over $R$} if for every
finite coloring of $R$, the system has a monochromatic solution.
\end{defn}

Combinatorialists have been working on this problem ever since Rado's
original result. Bergelson, Deuber, Hindman and Lefmann \cite{cite3}
showed that the result holds for a certain class of integral domains,
namely, the class of integral domains that are not fields and have
the property that $R/(r)$ is finite for every $r\neq0$. Byszewski
and Krawczyk \cite{cite5} then extended the result to all integral domains
and also considered the case where $b$ has entries in some $R$-module
and we are coloring the module instead. Finally, in 2020, Leader
and Russell \cite{cite15} showed that the same characterization holds
over any commutative ring $R$.

\begin{manualthm}{}[Leader-Russell]
Let $R$ be a commutative ring, $A$ an $m\times n$
matrix with entries in $R$ and $b\in R^{m}$, $b\neq0$. The system
$Ax=b$ is partition regular over $R$ if and only if it has a constant
solution.
\end{manualthm}

Among all of the aforementioned papers, including the newest one by Leader and Russell, the result by Straus \cite{cite24} appears as a common theme. 
This result is used in most of these papers, either directly or as a motivation behind some of the proofs. 

\subsection{Computability Theory and Reverse Mathematics}

In this section, we are going to present the computability-theoretic
and reverse-mathematical background that we need for this paper. For
a deeper dive into computability, we refer the reader to \cite{cite19} and \cite{cite22},
and for reverse mathematics \cite{cite21} and \cite{cite23}.

When we say \textit{computable functions}/\textit{sets}, we mean functions/sets that can be computed by an algorithm in a finite amount of time. \textit{Computable algebraic structures} (groups, rings) in this paper will refer to algebraic structures for which both the underlying set and the algebraic operations are computable.

 The \textit{Turing degrees} are a measure of complexity in computability theory -- you can think
of two sets having the same Turing degree (or being Turing equivalent)
if they are ``equally hard'' to compute. Among the Turing degrees,
PA degrees are the ones that compute infinite paths on infinite computable
binary (or $k$-branching for a fixed positive $k\in\mathbb{Z}$)
trees. Here, by a \textit{tree} $T$ we mean a subset of $2^{< \mathbb{N}}$ (or $\mathbb{N}^{< \mathbb{N}}$) such that if $\sigma \in T$, then every initial segment of $\sigma$ is in $T$.
\begin{defn}
A set $X$ has \textit{PA degree} if every computable infinite binary
tree has an $X$-computable path.
\end{defn}

There are many equivalent statements of the PA degrees and the one we are going to use directly in Theorem 3.1 is the following.
\begin{prop}
A degree is PA if and only if it computes a total $\{0,1\}$-valued extension of the $\{0,1\}$-valued function $e\mapsto\Phi_{e}(e)$. 
\end{prop}

Here and throughout the paper, $\Phi_{e}$ is the $e$th partial computable function. 

For Theorem 3.2 and Theorem 3.3, we use a result by Jockusch \cite{cite13} about the $\text{DNC}_{k}$ degrees. 
We include the proof since we need to explore the reverse mathematics of this statement for some of our theorems in Chapter 4.

\begin{defn}
A function $f:\mathbb{N}\rightarrow\mathbb{N}$ is called \textit{diagonally non-computable} 
(DNC) if for all $e$, $f(e)\neq\Phi_{e}(e)$, where $\Phi_{e}$ is
the $e$th partial computable function. By $\text{DNC}_{k}$ we denote
the class of all $k$-bounded DNC functions, i.e. all $f\in\text{DNC}$
such that $f(e)<k$ for all $e$.
\end{defn}
\begin{manualthm}{}[Jockusch]
Let $k\geq2$. The degrees of functions in $\text{DNC}_{k}$
coincide with the PA degrees. 
\end{manualthm}

\begin{proof}
It is well-known that the PA degrees coincide with the $\textup{DNC}_2$ degrees, so we will show that the $\textup{DNC}_2$ degrees coincide with the $\textup{DNC}_k$ degrees for every $k\in\mathbb{Z}, k>2$. 
One direction is clear since every $\textup{DNC}_2$ function is a $\textup{DNC}_k$ function, for $k>2$. For the other direction, it is enough to show that every $\textup{DNC}_{k^2}$ function computes a $\textup{DNC}_k$ function. Let $g\in\textup{DNC}_{k^2}$ and let $\langle \ ,\ \rangle:\mathbb{N}\times \mathbb{N} \rightarrow \mathbb{N}$ be any pairing function that takes $k\times k$ to $k^2$. For $a,b<k$ compute $\Phi_a(a), \Phi_b(b)$ and $\langle \Phi_a(a),\Phi_b(b) \rangle$. Find a $c$ such that $\Phi_c(c)=\langle \Phi_a(a),\Phi_b(b) \rangle$. For this $c$, find functions $g_1$ and $g_2$ such that $g(c)=\langle g_1(a,b), g_2(a,b) \rangle$. Since $g(c)\neq \Phi_c(c)$, either $g_1(a,b) \neq \Phi_a(a)$ or $g_2(a,b) \neq \Phi_b(b)$. There are two possible cases:
\begin{description}
\item[Case 1] For every $a$, there is a $b$ such that $g_2(a,b)=\Phi_b(b)$. Given $a$, find such a $b$ and define $h(a)=g_1(a,b)$.

\item[Case 2] There is an $a$ such that for all $b$ we have $g_2(a,b)\neq \Phi_b(b)$. Fix such an $a$ and define $h(b)=g_2(a,b)$.

\end{description}
\noindent The function $h$ defined in this way can be computed from $g$ and $h\in \textup{DNC}_k$.
\end{proof}

 In reverse mathematics, we gauge the strength of theorems over a base system of axioms in second-order arithmetic. The usual axiom system
is $\textup{RCA}_{0}$, which roughly corresponds to computable
mathematics. A simple result that does not hold computably but is at the combinatorial core of many theorems is Weak K{\"o}nig's Lemma (WKL).

\begin{lem}[Weak K{\"o}nig's Lemma]

Every infinite binary tree has an infinite path.

\end{lem}

This result is so significant in reverse mathematics that the next usual base system of axioms is $\textup{WKL}_{0}$, which consists of  $\textup{RCA}_{0}$ together with Weak K{\"o}nig's Lemma. From the statement, you can clearly see that there is a close connection between PA degrees and $\textup{WKL}_{0}$, which will be also demonstrated in our theorems.

A fact that we are going to use in Theorem 4.1 is the fact that $\textup{WKL}_0$ is equivalent to the $\Sigma_{1}^{0}$-\textit{separation principle}.

\begin{prop}[$\Sigma_{1}^{0}$-separation
principle]

Let $\varphi_0(n)$ and $\varphi_1(n)$ be $\Sigma_{1}^{0}$ formulas in which $X$ does not occur free. If $\lnot\exists n(\varphi_{0}(n)\wedge\varphi_{1}(n))$, then $$\exists X[(\varphi_0(n) \rightarrow n\notin X) \land (\varphi_1(n) \rightarrow n \in X)].$$
\end{prop}

\begin{prop}
Over $\textup{RCA}_0$, $\textup{WKL}_0$ is equivalent to the $\Sigma_{1}^{0}$-separation principle.
\end{prop}

For Theorem 4.3 and 4.4 we will need propositions about the reverse mathematics of $\textup{DNC}_k$ functions. It is well known that over $\textup{RCA}_0$, $\textup{WKL}_0$ is equivalent to the existence of $\textup{DNC}_2$ functions and the existence of PA degrees, so the following proposition shows that $\textup{WKL}_0$ is equivalent to the existence of $\textup{DNC}_k$ functions for a fixed $k\in\mathbb{N}, k\geq 2$, over $\textup{RCA}_0$.

\begin{lem} Jockusch's theorem is provable in $\textup{RCA}_0$.
\end{lem}

\begin{proof}
It is not too hard to see that the proof of Jockusch's theorem above can be carried out in $\textup{RCA}_0$.
\end{proof}

\begin{lem}
Over $\textup{RCA}_0$, $\textup{WKL}_0$ is equivalent to the statement $(\exists k)\textup{DNC}_k$.
\end{lem}

\begin{proof}
One direction is clear. For the other direction, assume that there exists a $\textup{DNC}_k$ function $f$ for some $k\in\mathbb{N}$. There is a number of the form $2^{2^m}\geq k$ for some $m\in\mathbb{N}$. Note that if $f\in\textup{DNC}_k$, then $f\in\textup{DNC}_{2^{2^{m}}}$. Following the proof of Jockusch's theorem, we obtain a $\textup{DNC}_{2^{2^m-1}}$ function in one of two possible cases, even though we do not necessarily know which case gives us this function. 

Form a binary tree where the level $0$ node is the function $f$, and each branch represents one of the two cases in Jockusch's theorem. Repeating the process, this tree will have $2^m$ levels and the last level will give us a $\textup{DNC}_2$ function. Using $\Pi_0^1$-induction we can show that at every level of the tree, there is a $\textup{DNC}_l$ function for an appropriate $l\in \mathbb{N}$. Hence the lemma follows from Lemma 2.2, the fact that this proof can be carried out in $\textup{RCA}_0$ and the fact that the $\textup{DNC}_2$ degrees are equivalent to $\textup{WKL}_0$ over $\textup{RCA}_0$.
\end{proof}

\vspace{1cm}

\section{The Computability of Straus' theorem}

\vspace{5mm}

In this section, we analyze Straus' theorem from a computability-theoretic point of view. In Theorem 3.1 we show that it does not hold computably and we find that the PA degrees are the best possible bound for the case $n=1$. For the case $n\geq 2$, in Theorem 3.2 we show that the PA degrees are a lower bound for the same problem with one additional condition added. Theorem 3.3 is a generalized version of Theorem 3.2 that works for any finite coloring at the same time, regardless of the number of colors. 

In all of these theorems, we build a computable abelian group $G$ that is isomorphic to ${\mathbb{Z}}^{(\omega)}$, where by ${\mathbb{Z}}^{(\omega)}$ we denote the additive group of all infinite sequences of integers, where each sequence has at most finitely many nonzero coordinates. We will
represent the elements of ${\mathbb{Z}}^{(\omega)}$ as finite strings of integers ignoring the trailing
zeroes, and represent the zero sequence by $0$. For example, we will
write $10^{3}1$ for the infinite sequence $(1,0,0,0,1,0,0,0,0,....)$.

During each stage of these constructions, we will have finitely many elements $G_{s}$ of $G$ and possibly an injective mapping $f:G_{s} \rightarrow {\mathbb{Z}}^{(\omega)}$. We want to define the addition on $G_s$ according to the addition in ${\mathbb{Z}}^{(\omega)}$ using the mapping $f$. In other words, for all $m,n \in G_s$, we want to define $m+n=f^{-1}(f(m)+f(n))$. However, at different stages, we will have different mappings so it is crucial to preserve the addition on $G_s$ that we have previously defined. That means that if $h:G_{s} \rightarrow {\mathbb{Z}}^{(\omega)}$ is an injective mapping used at a later stage, we require $f^{-1}(f(m)+f(n)) = h^{-1}(h(m)+h(n))$. This is illustrated in the following diagram.

\begin{figure}[h]
\caption{A visual representation of Lemma 3.0.1}
\centering
\includegraphics[width=14cm,scale=0.7]{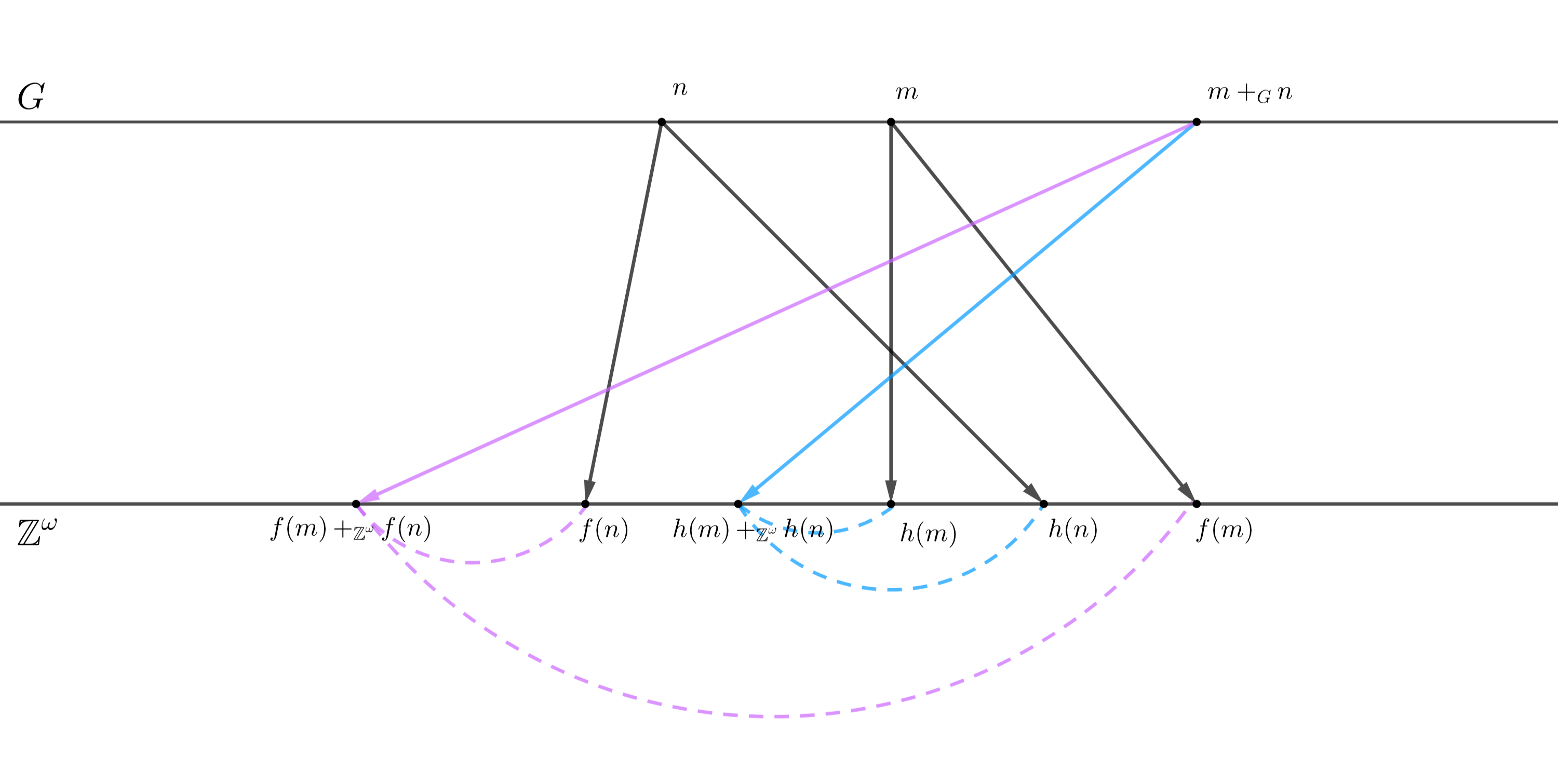}
\end{figure}

Choosing the mappings as in the following lemma is enough for this condition to be satisfied. Note that, since $(a_{1},a_{2},...)\in\mathbb{\mathbb{Z}}^{(\omega)}$,
it has finitely many nonzero coordinates, so any linear combination
of the $a_i$'s is going to be finite. 

\begin{lem}
Let $G$ be a nonempty set and
${f}:G\rightarrow\mathbb{\mathbb{Z}}^{(\omega)}$ an injective mapping. Let $f(m)=(a_1,a_2,...)$ and let $h:G\rightarrow\mathbb{\mathbb{Z}}^{(\omega)}$ be an injective mapping defined in the following way:
$$h(m)=(\underset{i}{\sum}k_{1i}a_{i},\underset{i}{\sum}k_{2i}a_{i},...,\underset{i}{\sum}k_{ji}a_{i},...)$$
for all $m\in G$, where the $k_{ji}\in\mathbb{Z}$ are fixed for all $i,j\in \mathbb{N}$. Then $h^{-1}(h(m)+h(n))=f^{-1}(f(m)+f(n))$
for all $m,n\in G$.
\end{lem}
\begin{proof}
We will
show that the 
 holds for $h$  defined in the following way $h(m)=h(f^{-1}(a_{1},a_{2},...))=(\underset{i}{\sum}k_{i}a_{i},0,0,0,...)$
for all $m\in G$, where the $k_{i}\in\mathbb{Z}$ are fixed. The general claim will
follow in a similar way. Let $m,n\in G$ and $f(m)=(a_{1},a_{2},...)$,
$f(n)=(b_{1},b_{2},...)$. We have

\begin{equation*}
\begin{split} 
h(m+n) & =h(f^{-1}(f(m)+f(n)))\\
&=h(f^{-1}(a_{1}+b_{1},a_{2}+b_{2},...))\\
&=(\underset{i}{\sum}k_{i}(a_{i}+b_{i}),0,0,0,...)\\
&=(\underset{i}{\sum}k_{i}a_{i},0,0,0,...)+(\underset{i}{\sum}k_{i}b_{i},0,0,0,...)\\
&=h(m)+h(n).
\end{split}
\end{equation*}

\noindent Since $h$ is injective, the claim follows.
\end{proof}

We arrive at the first theorem of this paper. The proof of this theorem follows a simple idea: if we have a $2$-coloring of a group $G$ with no monochromatic solutions to $x-y=b$, then any two $x_1,y_1$ such that $|x_1-y_1|=b$ have to be colored in a different color (Figure 2). It follows that, if $|x_1-y_1|=kb$ for some even integer $k$, then $x_1$ and $y_1$ have to be colored in the same color. Similarly, if $|x_1-y_1|=kb$ for some odd integer $k$, then $x_1$ and $y_1$ have to be colored in a different color. Following this idea, we will ensure that whenever the $e$th p.c. function halts, we have some $x_e$ and $y_e$ that are $kb$ apart for some integer $k$. This is reflected in the requirements $R_e$ below.

\begin{figure}[h]
\centering
\caption{The behavior of a $2$-coloring with no monochromatic solutions to the equation $x-y=b$}
\includegraphics[width=14cm,scale=0.7]{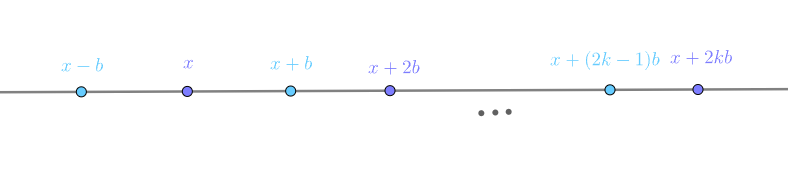}
\end{figure}

\begin{manualthm}{3.1}

There is a computable abelian group $(G,+,0)$ and an element $b\in G,b\neq0$
such that, if $c$ is a 2-coloring of $G$ for which the equation 
\[
x-y=b
\]
 has no monochromatic solutions, then $c$ has PA degree.
\end{manualthm}
\begin{proof}
We are going to build such a group $G$ that will be isomorphic to $\mathbb{\mathbb{Z}}^{(\omega)}$
in stages. At each stage $s$, we will build a finite set $G_{s}$
and an injective map $h_{s+1}:G_{s+1}\rightarrow\mathbb{Z}^{(\omega)}$
such that $h_{s}^{-1}(h_{s}(i)+h_{s}(j))=h_{s+1}^{-1}(h_{s+1}(i)+h_{s+1}(j))$
for all $i,j\in G_{s}$. Then $G=\bigcup G_{s}$ and we will show $\underset{s\rightarrow\infty}{\lim}h_{s}(x)$
exists so $h(x)=\underset{s\rightarrow\infty}{\lim}h_{s}(x)$ is going
to be an injective map from $G$ to $\mathbb{\mathbb{Z}}^{(\omega)}$.

While we are building $G$ and $h$, we want to satisfy certain requirements,
namely, ensuring that $G$ is a computable group (we will denote these
requirements by $A_{e}$ for addition and $I_{e}$ for inverses)
and that this group satisfies the conditions of the theorem (we will
denote these requirements by $R_{e}$). Before the construction, we
fix an element $b\neq0$ that will be mapped to $1\in\mathbb{\mathbb{Z}}^{(\omega)}$
during the construction. Also, for every $e$, we fix $2$ witnesses
for $R_{e}$, denoted by $x_{e}$ and $y_{e}$. At the first stage
$s$ where we define $h_{s}(x_{e})$ and $h_{s}(y_{e})$, we define
them by $0^{2e+1}1$ and $0^{2e+2}1$, respectively, and this will
change only when we redefine $h$ for the purposes of $R_{e}$. Let $\{\Phi_{e}\}_{e \in \mathbb{N}}$ be the standard enumeration of the $2$-valued p.c. functions.
The requirements that we want to satisfy for all $e\geq0$ are the
following:

\vspace{3mm}

 \makebox[1.5cm]{$A_{e}$:}  If $e=\langle i,j\rangle$, there is an $l$ such that $h(l)=h(i)+h(j)$

\vspace{2mm}

 \makebox[1.5cm]{$I_{e}$:}  There is an element $-e$ such that $h(e)+h(-e)=0$

\vspace{2mm}

 \makebox[1.5cm]{$R_{e}$:}  If $\Phi_{e}(e)\downarrow=0$ then $h(y_{e})=h(x_{e})+kh(b)$
for some even integer $k$ and if $\Phi_{e}(e)\downarrow=1$ then $h(y_{e})=h(x_{e})+kh(b)$ for some odd integer $k$.

\vspace{3mm}

\noindent Construction:

Stage $0$: Let $G_{0}=\{0,b\}$, $h_{0}(0)=0$ and $h_{0}(b)=1$.

Stage $s+1$:
\begin{itemize}
\item If $s=\langle e,3t+1\rangle$ for some $e,t\in\mathbb{N}$ and $e=\langle i,j\rangle$
for some $i,j\in\mathbb{N}$, check whether there is an $l$ such
that $h_{s}(l)=h_{s}(i)+h_{s}(j)$. 
\begin{itemize}
\item If there is such an $l$, let $G_{s+1}=G_{s}$ and $h_{s+1}(g)=h_{s}(g)$
for all $g\in G_{s}$. 
\item If there is no such $l$, let $l$ be the least natural number that
has not been considered so far in the construction and set $G_{s+1}=G_{s}\cup\{l\}$,
$h_{s+1}(l)=h_{s}(i)+h_{s}(j)$, and $h_{s+1}(g)=h_{s}(g)$ for all
$g\in G_{s}$.
\end{itemize}
\item If $s=\langle e,3t+2\rangle$ for some $e,t\in\mathbb{N}$, check
whether there is a $-e$ such that $h_{s}(e)+h_{s}(-e)=h_{s}(0)$. 
\begin{itemize}
\item If yes, let $G_{s+1}=G_{s}$ and $h_{s+1}(g)=h_{s}(g)$ for all $g\in G_{s}$. 
\item If there is no such $-e$, let $-e$ be the least natural number that
has not been considered so far in the construction and set $G_{s+1}=G_{s}\cup\{-e\}$,
$h_{s+1}(-e)=-h_{s}(e)$ and $h_{s+1}(g)=h_{s}(g)$ for all $g\in G_{s}$.
\end{itemize}
\item If $s=\langle e,3t+3\rangle$ for some $e,t\in\mathbb{N}$ and
\begin{itemize}
\item If $h_{s}(x_{e})$ or $h_{s}(y_{e})$ has not been defined, then let
$h_{s+1}(x_{e})=0^{2e+1}1$ and $h_{s+1}(y_{e})=0^{2e+2}1$, $h_{s+1}(g)=h_{s}(g)$
for all $g\in G_{s}$ and let $G_{s+1}=G_{s}\cup\{x_{e},y_{e}\}$.
\item If $h_{s}(x_{e})$ and $h_{s}(y_{e})$ have been defined and $\Phi_{e}(e)[s]\downarrow=0$,
let $k$ be a fresh large even number such that  $k>|(h_s({g_1}))_1|+|(h_s(g_2))_1|$ for all $g_1, g_2 \in G_{s}$ , where $(h_s({g}))_1$ is the first coordinate of $h_s(g)$. Then let $G_{s+1}=G_{s}$ and for all $g\in G_{s}$,
define
\begin{equation*}
h_{s+1}(g)=h_{s+1}(h_{s}^{-1}(a_{1},a_{2},...))=(a_{2e+3}k+a_{1}, a_{2},...,a_{2e+2}+a_{2e+3}, 0, a_{2e+4},...),
\end{equation*}
where $a_{2e+2}+a_{2e+3}$ is in the $2e+2$ position. If, instead,$\Phi_{e}(e)[s]\downarrow=1$,
choose $k$ to be a fresh large odd number such that $k>|(h_s({g_1}))_1|+|(h_s(g_2))_1|$ for all $g_1, g_2 \in G_{s}$ . Define $h_{s+1}(g)$
for all $g\in G_{s}$ as above. 
\end{itemize}
\end{itemize}

\vspace{3mm}
The following claims show that this construction gives a group $G$
that satisfies the theorem.

\begin{claim}
For all $s\in \mathbb{N}$, if $h_s$ is injective, then $h_{s+1}$ is injective.
\begin{proof}
Let $h_s$ be injective and $g_1,g_2\in G$. Let $h_s(g_1)=(a_1,a_2,...)$ and $h_s(g_2)=(b_1,b_2,...)$. Then $h_{s+1}(g_{1})=(a_{2e+3}k+a_{1}, a_{2},...,a_{2e+2}+a_{2e+3}, 0, a_{2e+4},...)$ and $h_{s+1}(g_{2})=(b_{2e+3}k+b_{1}, b_{2},...,b_{2e+2}+b_{2e+3}, 0, b_{2e+4},...)$. Assume $h_{s+1}(g_{1})=h_{s+1}(g_{2})$. Then all the coordinates of $h_s(g_{1})$ and $h_s(g_{2})$ are equal, except possibly the first one, the $2e+2$ and the $2e+3$ one. We also have $a_{2e+3}k+a_{1} =b_{2e+3}k+b_{1}$ and $ a_{2e+2}+a_{2e+3}=b_{2e+2}+b_{2e+3}$. From there we get
\begin{equation*}
\begin{split}
a_{1}-b_{1} & = (b_{2e+2}-a_{2e+2})k\\
 & = (b_{2e+3}-a_{2e+3})k
\end{split}
\end{equation*} 

\noindent and by our choice of $k$, this is only possible if $a_1=b_1$, $a_{2e+2}=b_{2e+2}$ and $a_{2e+3}=b_{2e+3}$. Hence, $h_{s}(g_1)=h_s(g_2)$ and since $h_s$ is injective, we have $g_1=g_2$. This $h_{s+1}$ is injective.
\end{proof}
\end{claim}

\begin{claim}
$h_{s}^{-1}(h_{s}(i)+h_{s}(j))=h_{s+1}^{-1}(h_{s+1}(i)+h_{s+1}(j))$ for
all $i,j\in G_{s}$.
\begin{proof}
This follows directly from Lemma 3.0.1 and Claim 3.0.2.
\end{proof}
\end{claim}

\begin{claim}
$h(x)=\underset{s\rightarrow\infty}{\lim}h_{s}(x)$ exists for all
$x\in G$.
\begin{proof}
We will show that for every $x\in G$, there are only finitely many stages
$s$ such that $h_{s+1}(x)\neq h_{s}(x)$ and then the claim follows.
Denote the ith coordinate of $h_{s}(x)$ by $(h_{s}(x))_{i}$, for
$i,s\in\mathbb{N},i>0$. Notice that for any $x\in G$, $h_{s+1}(x)\neq h_{s}(x)$
can only happen at stages $s+1$ where we are considering some $R_{e}$,
and the only coordinates that can change at that stage are $(h_{s}(x))_{1},(h_{s}(x))_{2e+2}$
and $(h_{s}(x))_{2e+3}$. We are going to refer to $(h_{s}(x))_{2e+2}$
and $(h_{s}(x))_{2e+3}$ as the $e$-relevant coordinates. Notice
that if all the $e$-relevant coordinates are $0$ then $h_{s+1}(x)=h_{s}(x)$.

Let $J>0$ be the index of the
rightmost nonzero coordinate of $h_{s}(x)$. Then any $R_{e'}$ such that $2e'+2>J$ will not affect $h_{s}(x)$ since all the $e'$-relevant coordinates are $0$.  Hence there are only finitely many requirements that can change $h_{s}(x)$. Once a requirement is satisfied, we never consider it again, so $h_{s}(x)$ changes at most finitely many times.

\end{proof}
\end{claim}
\vspace{-5mm}
\begin{claim}
$h$ is injective.
\begin{proof}
This follows from Claim 3.0.3 and Claim 3.0.4.
\end{proof}
\end{claim}
\noindent \textup{Now, since}

$h_{s+1}(0)=h_{s+1}(0,0,0,...)=(0k+0,0,0,...,0+0,0,0...)=(0,0,0,...)$

$h_{s+1}(b)=h_{s+1}(h_{s}^{-1}(1,0,0,....))=(0k+1,0,0,...,0+0,0,0,...)=(1,0,0,....)$, 

\noindent \textup{the values of} $h(0)$ and $h(b)$ \textup{never change throughout
the construction, so we have} $h(0)=0,h(b)=1$.

\textup{Similarly, for every} $s$ \textup{we have,} $h_{s+1}(x_{e})=h_{s}(x_{e})$ \textup{and} $h_{s+1}(y_{e})=h_{s+1}(x_{e}) + kh_{s+1}(b)$ \textup{for some positive integer} $k$\textup{, and hence} $h(y_{e})=h(x_{e})+kh(b)$. 

\begin{claim}
G as defined above satisfies the conditions of the theorem.
\begin{proof}
$G$ is a computable abelian group since for all $i,j\in G$, $i+j,-i,-j$
are defined at some stage and to compute $i+j$, we only need to find
the first stage $s$ at which $h_{s}^{-1}(h_{s}(i)+h_{s}(j))$ is
defined. Let $c:G\rightarrow\{c_{1},c_{2}\}$ be a $2$-coloring of
$G$ with no monochromatic solutions to $x-y=b$. Define a function $g$ in the following way:
\[
g(e)=\begin{cases}
\begin{array}{cc}
0, & \text{if\ \ensuremath{c(x_{e})=c(y_{e})}}\\
1, & \text{if}\ \ensuremath{c(x_{e})\ne c(y_{e})}
\end{array}.\end{cases}
\]
If $\Phi_{e}(e)\downarrow$, then there is a stage $s$ at which
$\Phi_{e}(e)[s]\downarrow$. Hence there is an integer $k$
such that $h(y_{e})=h(x_{e})+kh(b)$. Since $c$ has no monochromatic
solutions to $x-y=b$, it follows that $c(x+b)\neq c(x)$ for any
$x\in G$. Thus, if $k$ is even, we have $c(x_{e})=c(x_{e}+kb)=c(y_{e})$
and hence, $g(e)=0=\Phi_{e}(e)$. Similarly, if $k$ is odd, it must
be the case that $c(x_{e})\neq c(y_{e})$ and hence, $g(e)=1=\Phi_{e}(e)$.
We conclude that $c$ computes a $\{0,1\}-$valued extension of the
function $e\mapsto\Phi_{e}(e)$, so $c$ has PA degree.
\end{proof}
\end{claim}
\textup{Thus we built a group} $G$ \textup{that is isomorphic to a subgroup of} $\mathbb{Z}^{(\omega)}$\textup{.
We claim that} $G$ \textup{is infinitely generated. Assume
otherwise, and let} $g_{1},...,g_{l}$ \textup{be the generators. Let} $L$
\textup{be the maximal nonzero coordinate in} $h(g_{1}),\dotsc,h(g_{l})$\textup{. By our construction,
we will eventually add an element} $g$ \textup{to our group such that the
maximal nonzero coordinate of} $h(g)$ \textup{is} $>L$\textup{. Then} $g$ \textup{cannot be written
as a linear combination of} $g_{1},...,g_{l}$\textup{. Hence} $G$ \textup{is infinitely generated and
since all infinitely generated subgroups of} $\mathbb{Z}^{(\omega)}$
\textup{are isomorphic to} $\mathbb{Z}^{(\omega)}$\textup{, we get that, in fact,
our group} $G$ \textup{is isomorphic to} $\mathbb{Z}^{(\omega)}$.
\end{proof}

Note that if the $2$-colorings of an abelian group $G$ are represented
by a binary tree, then we can search through the tree and cut all
branches where we find a monochromatic solution to $x-y=b$. Thus
we can find an infinite path in this tree that will represent a coloring
with no monochromatic solutions to the equation and PA degrees compute
these paths.  Similarly, since the PA degrees compute paths in $k$-branching trees for a fixed integer $k>2$, we get that 
the PA degrees are an upper bound for the case $n\geq2$. Hence the PA degrees are the \textit{best possible bound} for this
problem.

\begin{prop}
Let $n \in \mathbb{N}, n>0$. For every computable abelian group $G$ and every $b\in G, b\neq 0$, there is a $k$-coloring of PA degree with no pairwise monochromatic solution to

$$(x_{1}-y_{1})+...+(x_{n}-y_{n})=b,$$

\noindent where 

\[
k=\begin{cases}
\hspace{4mm}2n & \text{if }2\mid \ordb \text{ or } \ordb=\infty\\
\left\lceil \dfrac{2np}{p-1}\right\rceil  & \text{if }\ordb \text{ is odd and }p\text{ is the largest prime divisor of }\ordb
\end{cases}.
\]

\end{prop}

Theorem 3.1 gives the best possible result of this sort:
\begin{prop}
Let $n>2, n \in \mathbb{N}$. Then for every computable abelian group $G$, there is a computable $n$-coloring $c$ with no monochromatic solutions to  $x-y=b$, where $b\in G, b\neq 0$.

\begin{proof} 
Let $G=\{g_1,g_2,g_3,...\}$ be a computable abelian group. It is enough to show that the claim holds for $n=3$. Let $c_1,c_2,c_3$ be three different colors. Color $g_1$ with $c_1$. If $|g_1-g_2|=b$, color $g_2$ with $c_2$. If not, color $g_2$ with $c_1$. Continue in a similar manner. For $g_k$, check whether $|g_k-g_i|=b$ for all $i<k$. There are three situations that can occur:
\begin{itemize}
\item There is exactly one $i<k$ such that $|g_k-g_i|=b$. Then color $g_k$ with a different color than $g_i$. 
\item There is no such $i<k$. Then color $g_k$ with $c_1$.
\item There are two $i_1, i_2<k$ such that $|g_k-g_{i_1}|=b$ and $|g_k-g_{i_2}|=b$. In this case, color $g_k$ with a different color than $g_1$ and $g_2$. This is possible since we have 3 colors. 
\end{itemize}
\end{proof}
\end{prop}

As we mentioned previously, the same exact proof as in Theorem 3.1 applies for the Straus* theorem with the additional caveat that we also need a mapping from $G$ to $G$. We can take this mapping to be the identity mapping, and in particular, it is computable.

\begin{manualthm}{3.1*}

There is a computable abelian group $(G,+,0)$, an element $b\in G,b\neq0$ and a computable mapping $f_1:G\rightarrow G$
such that, if $c$ is a 2-coloring of $G$ for which the equation 
\[
f_1(x)-f_1(y)=b
\]
 has no monochromatic solutions, then $c$ has PA degree.
\end{manualthm}

The * version of Proposition 3.0.7 also holds, so we get that the PA degrees are the best possible computability-theoretic bounds for the Straus* theorem in the case $n=1$.

If we add the additional condition that
there is an $m>0$ such
that for every $k\in\mathbb{Z}$, all the $kmb$ are colored with
the same color, we get that the PA degrees are a lower bound for the case $n\geq2$. The idea of the proof is similar to the case $n=1$. By adding this condition, we know the behaviour of any $2n$-coloring with no pairwise monochromatic solutions to $(x_{1}-y_{1})+\dots+(x_{n}-y_{n})=b$ (Figure 3). Namely, if $x_1$ and $y_1$ are such that $|x_1-y_1|=(km+1)b$ for some $k\in\mathbb{Z}$, then they have to be colored with a different color. Otherwise, these $x_1$ and $y_1$ together with any $x_2$ and $y_2$ such that $|x_2-y_2|=kmb$ are going to give a pairwise monochromatic solution to the equation $(x_{1}-y_{1})+\dots+(x_{n}-y_{n})=b$ for any $n\geq2$.

\begin{figure}[h]
\caption{The behavior of a $2n$-coloring with no pairwise monochromatic solutions to the equation $(x_{1}-y_{1})+\dots+(x_{n}-y_{n})=b$ under an additional condition}
\centering
\includegraphics[width=14cm,scale=0.7]{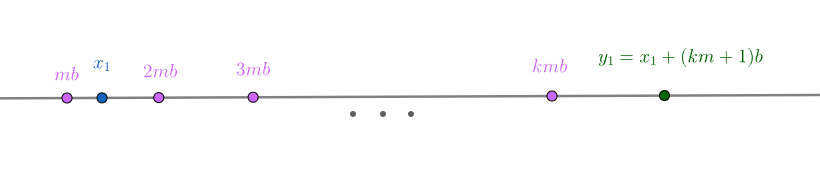}
\end{figure}

\begin{manualthm}{3.2}
Let $n\in\mathbb{N},n\geq2$. There is a computable group $(G,+,0)$
and an element $b\in G,b\neq0$ such that, if $c$ is a $2n$-coloring
of $G$ for which the equation 
\begin{equation}
(x_{1}-y_{1})+...+(x_{n}-y_{n})=b
\end{equation}
has no pairwise monochromatic solutions and there is an $m>0$ such
that for every $k\in\mathbb{Z}$, all the $kmb$ are colored with
the same color, then $c$ has PA degree.
\end{manualthm}
\begin{proof}
Similarly to the previous theorem, we are going to build such a group
$G$ in stages and at every stage $s$, we will define $G_{s}$ and
an injective mapping $h_{s}:G_{s}\rightarrow\mathbb{Z}^{(\omega)}$ such that $h_{s}^{-1}(h_{s}(i)+h_{s}(j))=h_{s+1}^{-1}(h_{s+1}(i)+h_{s+1}(j))$ for
all $i,j\in G_{s}$. Then $G=\bigcup G_{s}$. At the beginning
of the construction, we fix an element $b\neq0$ that will be sent
to $1\in\mathbb{\mathbb{Z}}^{(\omega)}$ during the construction.
Also, for every $e$, we fix $2n+1$ witnesses $x_{ei}, i=1,\dots,2n+1$, that will be sent to $0^{(e-1)(2n+1)+i-1}1$ during the construction
at the first stage they are defined. Let $\{\Phi_{e}\}_{e \in \mathbb{N}}$ be the standard enumeration of the $\binom{2n+1}{2}$-valued p.c. functions. Since we want to ensure that
the group we are building is computable, the requirements $A_{e}$
and $I_{e}$ are exactly as before:
\vspace{3mm}

\makebox[1.5cm]{$A_{e}$:}  If $e=\langle i,j\rangle$, there is an $l$ such that $h(l)=h(i)+h(j)$

\vspace{2mm}

\makebox[1.5cm]{$I_{e}$:}  There is an element $-e$ such that $h(e)+h(-e)=h(0)$

\vspace{2mm}

 \makebox[1.5cm]{$R_{e}$:}  Fix a bijection from $\{0,1,\dots,\binom{2n+1}{2}\}$ to $\{(x_{ei},x_{ej})|1\leq i \leq j \leq 2n+1\}$. Let $(x_{ei_{1}},x_{ei_{2}})$ be the image of $i$ under this bijection. If $\Phi_{e}(e)\downarrow=i$, then there is an integer $k$ such that $h(x_{ei_{2}})=h(x_{ei_{1}})+(mk+1)h(b)$.

\vspace{3mm}

\noindent Construction:

Stage $0$: Let $G_{0}=\{0,b\}$, $h_{0}(0)=0$ and $h_{0}(b)=1$.

Stage $s+1$:
\begin{itemize}
\item If $s=\langle e,3t+1\rangle$ or $s=\langle e,3t+2\rangle$ for some $e,t\in\mathbb{N}$
and $e=\langle i,j\rangle$ for some $i,j\in\mathbb{N}$, we proceed in the same
manner as before.
\item If $s=\langle e,3t+3\rangle$ for some $e,t\in\mathbb{N}$ and
\begin{itemize}
\item If there exists a $j\in\{1,2,\dots,2n+1\}$ such that $h_{s}(x_{ej})$
has not been defined, then let $h_{s+1}(x_{ej})=0^{(e-1)(2n+1)+j-1}1$ for
all such $j$ and $G_{s+1}=G_{s}\cup\{x_{e1},\dots,x_{e(2n+1)}\}$. 
\item If $h_{s}(x_{ej})$ has been defined for all $j\in\{1,2,\dots,2n+1\}$
and $\Phi_{e}(e)[s]\downarrow=i$ for some $i\in\{1,2,\dots,\binom{2n+1}{2}\}$, let $j_{1}$ and $j_{2}$ be the positions of the $1$ in $x_{ei_{1}}$ and $x_{ei_{2}}$,
respectively. Let $k$ be a fresh large number $k>|(h_s({g_1}))_1|+|(h_s(g_2))_1|$ for all $g_1, g_2 \in G_{s}$ , where $(h_s({g}))_1$ is the first coordinate of $h_s(g)$. Then for all $g\in G_{s}$, if $h_s(g)=(a_1,a_2,\dots)$, define
$$h_{s+1}(g)=((mk+1)a_{j{2}}+a_{1},a_{2},\dots,a_{j_{1}}+a_{j_{2}},a_{j_{1}+1},\dots,0,a_{j_{2}+1},\dots),$$
where $a_{j_{1}}+a_{j_{2}}$ is at the $j_{1}$
position and the $j_{2}$ position is $0$.
Let $G_{s+1}=G_{s}$.
\end{itemize}
\end{itemize}
\vspace{5mm}
The following claims are exactly as in the previous theorem and the proofs are very similar so we will omit them.

\begin{claim}
$h_{s}^{-1}(h_{s}(i)+h_{s}(j))=h_{s+1}^{-1}(h_{s+1}(i)+h_{s+1}(j))$ for
all $i,j\in G_{s}$.
\end{claim}

\begin{claim}
$h(x)=\underset{s\rightarrow\infty}{\lim}h_{s}(x)$ exists for all
$x\in G$.
\end{claim}

\begin{claim}
$h$ is injective.
\end{claim}

The only claim left to prove is the following:
\begin{claim}
$G=\bigcup G_{s}$ as defined by this construction satisfies the conditions
of the theorem.
\begin{proof}
$G$ is a computable abelian group for the same reasons as above.
Let $c$ be a coloring of $G$ with no pairwise monochromatic solutions
to $(1)$. Define a function $g:\mathbb{N}\rightarrow\{0,1,2,\dots,\binom{2n+1}{2}\}$ in the following way:
\begin{equation*}
g(e)=\text{the least }i\text{ such that }c(x_{ei_{1}})=c(x_{ei_{2}}).
\end{equation*}

Such an $i$ exists since for every $e$, we have $2n+1$ witnesses
and these are enough to ensure that we have a monochromatic pair
among them. Assume $\Phi_{e}(e)\downarrow=i$. Then there is a large $k$ such that $h(x_{ei_{2}})=h(x_{ei_{1}})+(mk+1)h(b)$. Now, by our assumption, $mkb$ and $2mkb$ are colored with the same color, so these, together with $x_{ei_{1}}$ and $x_{ei_{2}}$ give a pairwise monochromatic solution to $(1)$. Hence, for all $e$, if $\Phi_{e}(e)\downarrow$, then $\Phi_{e}(e)\neq g(e)$. Thus $g$ is a $\binom{2n+1}{2}$-bounded DNC function, so $c$ has PA degree.
\end{proof}
\end{claim}

\noindent Similarly to the previous theorem, we get that the constructed group
$G$ is isomorphic to $\mathbb{Z}^{(\omega)}$.
\end{proof}

As mentioned in the introduction, the additional condition that we added in Theorem 3.2 is actually shown in the original proof of Straus' theorem. This proof heavily relies on the algebraic structure of the group and Theorem 3.2 suggests that any other similar algebraic proofs will also give us PA degrees. However, it would be very interesting to see whether there is a combinatorial proof, which does not rely on the algebra, that allows us to build simpler witnesses for Straus' original theorem. 

Similarly to Theorem 3.1*, by taking $f_1,f_2,\dots,f_n$ to be the identity mapping, we obtain * versions of Theorem 3.2 and Theorem 3.3.

\begin{manualthm}{3.2*}
Let $n\in\mathbb{N},n\geq2$. There is a computable group $(G,+,0)$, an element $b\in G,b\neq0$ and computable mappings $f_1,f_2,\dots,f_n$ such that, if $c$ is a $2n$-coloring
of $G$ for which the equation 
$$(f_1(x_{1})-f_1(y_{1}))+\dots+(f_n(x_{n})-f_n(y_{n}))=b$$
has no pairwise monochromatic solutions and there is an $m>0$ such
that for every $k\in\mathbb{Z}$, all the $kmb$ are colored with
the same color, then $c$ has PA degree.
\end{manualthm}

\begin{manualthm}{3.3*}
Let $n\in\mathbb{Z},n\geq2$. There is a computable abelian
group $(G,+,0)$, an element $b\in G,b\neq0$ and computable mappings $f_1,f_2,\dots,f_n$ such that, if $c$ is a finite coloring
of $G$ for which the equation 
$$(f_1(x_{1})-f_1(y_{1}))+\dots+(f_n(x_{n})-f_n(y_{n}))=b$$
has no pairwise monochromatic solutions and there is an $m>0$ such
that for every $k\in\mathbb{Z}$, all the $kmb$ are colored with
the same color, then $c$ has PA degree. 
\end{manualthm}

The PA degrees are again the best possible bounds for Straus* in the case $n\geq 2$ with this additional condition added.

\vspace{1cm}

\section{The Reverse Mathematics of Straus' Theorem}

\vspace{5mm}

In this section, we will investigate the reverse mathematics of Straus' theorem. The PA degrees as bounds in the previous theorems already indicate that Straus' theorem might be equivalent to $\textup{WKL}_0$ over $\textup{RCA}_{0}$. Indeed, we show that the case $n=1$ is equivalent to $\textup{WKL}_0$ (Theorem 4.1) and we show one direction of the equivalence in the case $n\geq 2$ (Theorem 4.2). Then, using the reverse mathematics of Jockusch's theorem, we give a reverse-mathematical version of Theorem 3.2 and Theorem 3.3.

\begin{manualthm}{4.1}
Over $\textup{RCA}_{0}$, $\textup{WKL}_0$ is equivalent to the following
statement:

 For every abelian group $G$ and every element $b\in G,b\neq0$,
there is a $k-$coloring of $G$ for which the equation $x-y=b$ has
no monochromatic solutions, where $k=2$ if $\ordb=\infty$ or $\ordb$
is even and $k=3$ if $\ordb$ is odd.
\end{manualthm}
\begin{proof}

Assume $\textup{WKL}_{0}$. Let $G=\{a_{1},a_{2},\dots\}$ be an abelian
group and let $b\in G,b\neq0$ be an element of infinite order. Consider
the tree $T$ of all $2-$colorings of $G$ with no monochromatic
solutions to $x-y=b$, where every level of the tree represents an element of $G$. We need to show that this tree is infinite,
and then by an application of Weak K{\"o}nig's lemma, the statement $\textup{(ii)}$ follows for the case when $\ordb=\infty$.
For $m\in\mathbb{N},m>0$, consider the set $\{a_{1},\dots,a_{m}\}$ and
consider the following graph on $m$ vertices. Let the vertices of the graph be $\{a_{1},\dots,a_{m}\}$ and say $a_{i}$ and $a_{j}$ for $i,j\in\{1,\dots,m\}$
are connected if and only if $|a_{i}-a_{j}|=b$. Notice that for every $a_{i},i\in\{1,\dots,m\}$,
there are at most two elements $a_{j},a_{k}\in\{a_{1},a_{1},\dots,a_{m}\}\backslash\{a_{i}\}$
that are connected to it. We will show that this graph contains no loops.
Assume $a_{i_{1}},\dots,a_{i_{l}}$ where $l,i_1,i_2,\dots,i_l\in\{1,\dots,m\}$ form a loop in this exact order ($a_{i_1}$ is connected to $a_{i_2}$, which is connected to $a_{i_3}$ and so on). Without loss of generality,  assume that $a_{i_{1}}-a_{i_{2}}=b$.
Then it must be the case that $a_{i_{2}}-a_{i_{3}}=b$, $a_{i_{3}}-a_{i_{4}}=b,\dots,a_{i_{l-1}}-a_{i_{l}}=b,a_{i_{l}}-a_{i_{1}}=b$.
This is a contradiction since if we add the left hand sides and right hand sides of these equations,
we get $0=lb$, but $b$ has infinite order. Now
we can $2-$color $a_{1},a_{2},\dots,a_{m}$ such that there are no
monochromatic solutions to $x-y=b$ by coloring any two connected
elements with different colors.  Thus $T$ has an element of length $m$ for every $m\in\mathbb{N}$,
so it is infinite. Hence, by Weak  K{\"o}nig's lemma, $T$ has an infinite path, i.e. a $2$-coloring of $G$ with no monochromatic solutions to $x-y=b$.

The cases when \ordb  \hspace{0.5mm} is finite are similar. In the even case, there
might be loops on an even amount of vertices, but this will not cause
a problem since we can always color such a loop with two colors so
that no two neighboring vertices have the same color. In the odd case,
there might be loops on an odd number of vertices, which can always
be colored with three colors so that no two neighboring vertices have
the same color.

For the other direction, assume $\textup{(ii)}$. Notice that in the proof of Theorem 3.1, we
showed that $h(x)=\underset{s\rightarrow\infty}{\lim}h_{s}(x)$
exists, but, even though it made the proof more convenient, that was
not necessary for the construction. With that in mind, it is not hard
to see that the whole proof can be carried out in $\textup{RCA}_{0}$.
Now, consider two $\Sigma_{1}^{0}$ formulas $\varphi_{0}(n)$ and $\varphi_{1}(n)$
in which $X$ does not occur freely, such that $\lnot\exists n(\varphi_{0}(n)\wedge\varphi_{1}(n))$.
We can use the same construction as in Theorem 3.1, but instead of acting
when $\Phi_{e}(e)$ converges for $R_{e}$, we act at stages when we see that $\varphi_0(e)$ or $\varphi_1(e)$ holds. Then the
set $X=\{e|c(x_{e})=c(y_{e})\}$ is a separator for $\varphi_{0}$
and $\varphi_{1}$ and hence $\textup{(ii)}$ implies the $\Sigma_{1}^{0}$-separation
principle, which is equivalent to $\textup{WKL}_{0}$.
\end{proof}

\begin{manualthm}{4.1*}
The following statements are equivalent over $\textup{RCA}_0$:
\begin{itemize}
\item[$\textup{(i)}$] $\textup{WKL}_0$
\item[$\textup{(ii)}$]  For every abelian group $G$, every element $b\in G,b\neq0$ and all mappings $f:G\rightarrow G$,
there is a $k$-coloring of $G$ for which the equation $f(x)-f(y)=b$ has
no monochromatic solutions, where
\[
k=\begin{cases}
\hspace{4mm}2 & \text{if }2\mid \ordb \text{ or } \ordb=\infty\\
\hspace{4mm}3   & \text{if }\ordb \text{ is odd }
\end{cases}.
\]
\end{itemize}
\end{manualthm}

It is worth noting that if $\ordb$ is finite, then $\textup{(ii)}$ in Theorem 4.1 holds computably, so we have the following proposition.

\begin{prop}
For every computable abelian group $G$ and every element $b\in G,b\neq0$ of finite order,
there is a computable $k$-coloring of $G$ for which the equation $x-y=b$ has
no monochromatic solutions, where
\[
k=\begin{cases}
\hspace{2mm}2 & \text{if } \ordb \text{ is even }\\
\hspace{2mm}3   & \text{if }\ordb \text{ is odd }
\end{cases}.
\]
\end{prop}

Next, we investigate the reverse mathematics of the full Straus' theorem. The reader might notice that the case $n=1$ is included in the following theorem, even though we already showed this implication in Theorem 4.1. However, the proofs of these theorems are different and the proof of Theorem 4.4 closely follows Straus' original proof sketched below. 

The proof of Theorem 4.4 suggests that any other proof that relies on the algebraic structure will be similar.  However, we cannot exclude the possibility that 
there might be a combinatorial proof that would indicate lower bounds for the case $n\geq2$, which would be very interesting to see.

\begin{manualthm}{4.2}
Over $\textup{RCA}_0$, $\textup{WKL}_{0}$ implies the following statement:

Let $n\in\mathbb{N},n\geq1$. For every abelian group $(G,+,0)$ and every $b\in G,b\neq0,$ there
is a $k$-coloring of $G$ such that the equation
\begin{equation*}
(x_{1}-y_{1})+\dots+(x_{n}-y_{n})=b
\end{equation*}

\noindent has no pairwise monochromatic solutions, where 
\[
k=\begin{cases}
\hspace{4mm}2n & \text{if }2\mid \ordb \text{ or } \ordb=\infty\\
\left\lceil \dfrac{2np}{p-1}\right\rceil  & \text{if }\ordb \text{ is odd and }p\text{ is the largest prime divisor of }\ordb
\end{cases}.
\]
\end{manualthm}
\begin{proof} We start with a sketch of the proof of Straus' theorem since our proof will follow the same idea. We will skip the algebraic details of the proof that we do not directly use. We refer the reader to \cite{cite8}, \cite{cite14}, \cite{cite20} for a deeper dive into these topics.
\vspace{2mm}

\noindent \textit{Sketch of the Proof of Straus' Theorem}. Let $B$ be a subgroup of $G$ that is maximal subject to not containing $b$, and form the quotient $G/B$. Let $\bar{b}$ be the class of $b$ in this quotient group. By maximality, every nontrivial subgroup of $G/B$ must contain $\bar{b}$. It follows that every element of $G$ has finite order. Thus $\bar{b}$ has prime order $p$ and $G/B$ is a $p$-group in which every finitely generated subgroup is cyclic. Hence, if $G/B$ is finite, it is cyclic of order $p^n$ and if it is infinite, it is isomorphic to  $\mathbb{Z}_{p^{\infty}}$. In either case, it is a subgroup of $\mathbb{Q}/\mathbb{Z}$.

If $p=2$, identify $\bar{b}$ with $\dfrac{1}{2}$ in $\mathbb{Q}/\mathbb{Z}$ and color $[0,1)$ by a $2n$-coloring $c$
so that each interval $\left[\dfrac{i-1}{2n},\dfrac{i}{2n}\right)$,
$i=1,2,\dotso,n$ is colored with a different color. Then if $c(\bar{x})=c(\bar{y})$,
we have $|\bar{x}-\bar{y}|<\dfrac{1}{2n}$ and hence if the equation
$\stackrel[i=1]{n}{\sum}(\bar{x}_{i}-\bar{y_{i}})=\bar{b}$ has such
a solution, it follows that

\[
\dfrac{1}{2}=|\bar{b}|\leq\sum_{i=1}^{n}|\bar{x}_{i}-\bar{y_{i}}|<\dfrac{1}{2},
\]

\noindent a contradiction. If $p>2$, identify $\bar{b}$ with $\dfrac{p-1}{2p}$
and color $[0,1)$ by a coloring $c$ with $\left\lceil \dfrac{2np}{p-1}\right\rceil $
colors so that each interval $\left[\dfrac{(i-1)(p-1)}{2np},\min\{\dfrac{i(p-1)}{2np},1\}\right)$,
$i=1,2,\dotso,n$ is colored with a different color. In this case,
too, we get a contradiction.
 The desired coloring of $G$ is the coloring where every element $x\in G$ is colored with the color of its class $\bar{x}$.

Notice that for every prime $p|\ordb$, we can choose $B$ to contain $pb$ as well as all elements whose order is prime to $p$, and since we are modding out by $B$, all of those elements will be colored with the same color.

\vspace{2mm}

We now proceed with the proof of our theorem. Consider the $k$-branching tree of all $k$-colorings of $G$ with no pairwise monochromatic solutions to equation (2.1),
where every level of the tree represents an element of $G$. We want
to show that this tree is infinite and then by an application of Weak K\"onig's Lemma,
the claim follows.

Let $G_{m}=\{0,b,pb,a_{3},\dots,a_{m-1}\}\subset G$ be closed under
inverses, where $p=2$ if $\ordb$ is even or infinite, and $p$ is the largest
prime divisor of \ordb, otherwise. We will call $G_{m}$ a partial
group, where the addition on $G_{m}$ is defined the same way as the
addition in $G$ and for all $x,y\in G_{m}$, $x+y$ is defined if
and only if $x+y\in G_{m}$. We extend $G_m$ to $G'_m$ to satisfy the following condition: for all $x_1,y_1,x_2,y_2\in G_m$, $(x_1-y_1)+(x_2-y_2)\in G'_m$. To achieve this,
start with $G_m$, close under addition once, then close under inverses and finally, close under addition again. Call this partial group $G'_m$. 
We say that $B$ is a partial subgroup
of $G_{m}$ if $B\subset G_{m}$ and is closed under the partial addition
in $G_{m}$. Clearly, if $B$ is a partial subgroup of $G_m$ then it is a partial subgroup of $G'_m$. We want to choose a partial subgroup $B$ that contains all multiples of $pb$, that is maximal
subject to not containing $b$ and such that for any $b_{1},b_{2},b_{3},b_{4}\in B$,
$(b_{1}-b_{2})+(b_{3}-b_{4})\neq b$. Note that there is at least
one such partial subgroup, the group of all multiples of $pb$, although it is not necessarily
maximal. Consider all such partial subgroups of $G_{m}$ and form
the partial quotients that we will denote by $G'_{m}/B$ in the usual
way. If $a\in G'_m$, denote the class of $a$ in the quotient by $\bar{a}$. We say that two classes $\bar{a}$,$\bar{b}\in G'_{m}/B$
are equal if they have a nonempty intersection. We check whether $G'_{m}/B$
is a partial group. Assume not. Then there are elements $x_{1},x_{2},y_{1},y_{2}\in G'_{m}$
such that $\bar{x}_{1}=\bar{x}_{2}$ and $\bar{y}_{1}=\bar{y}_{2}$ but
$\bar{x}_{1}+\bar{y}_{1}\neq\bar{x}_{2}+\bar{y}_{2}$. From this we
get that $x_{1}-x_{2}\in B$, $y_{1}-y_{2}\in B$, so $(x_{1}-x_{2})+(y_{1}-y_{2})$
is in the group generated by $B$ but $(x_{1}-x_{2})+(y_{1}-y_{2})\notin B$.
Since we want $b\notin B$, we have $(x_{1}-x_{2})+(y_{1}-y_{2})\neq b$. However, by our construction of $G'_m$, $(x_{1}-x_{2})+(y_{1}-y_{2})\in G'_m$ so $(x_{1}-x_{2})+(y_{1}-y_{2})$ together with some elements of $B$ satisfies $(b_{1}-b_{2})+(b_{3}-b_{4})=b$. Hence, we can discard $B$ as a partial subgroup that satisfies our conditions.
After we are left with all partial subgroups that satisfy the conditions,
choose one that is maximal. 
\begin{claim}
$\bar{b}$ is in every nontrivial partial subgroup of $G_{m}/B$.
\end{claim}
\begin{proof}
Let $M/B$ be a nontrivial partial subgroup of $G_{m}/B$ and $\bar{b}\notin M/B$.
Then $b\notin M$ and $B\subset M$, which contradicts the maximality
of $B$.
\end{proof}
\begin{claim}
$\bar{b}$ has finite order.
\end{claim}
\begin{proof}
Since we chose $B$ such that $pb \in B$, we have $p\bar{b}=\bar{0}$ and hence $\bar{b}$ has finite order.
\end{proof}
\begin{claim}
Every element of $G_{m}/B$ has finite order.
\end{claim}
\begin{proof}
If $\bar{a}\neq\bar{0}\in G_{m}/B$, then $\bar{b}\in\langle\bar{a}\rangle$.
Hence there is an $k\in\mathbb{Z}\backslash\{0\}$ such that $\bar{b}=k\bar{a}$,
and since $\bar{b}$ has finite order, $\bar{a}$ has finite order.
\end{proof}
Now extend $G_{m}/B$ by taking all linear combinations of elements
in $G_{m}/B$ and call the extension $G_{m}^{*}/B$. This is possible
since every element of $G_{m}/B$ has finite order. Note that this does not change the class of $b$, except by possibly adding some elements to it. If $G_{m}^{*}/B$
is a group, proceed with the proof. If not, we discard $B$ as a possible candidate
and start the construction all over until we get a $B$ which satisfies
all of our conditions including $G_{m}^{*}/B$ being a group.
\begin{claim}
$\textup{RCA}_{0}$ proves the following theorem: Let $p$ be a prime
number and $G$ a group such that $|G|=n$. Then $p|n$ if and only
if there is an element $g$ of $G$ (and hence a subgroup) such that
$ord(g)=p$. 
\end{claim}
\begin{proof}
This statement follows from theorems by Lagrange and Cauchy that are standard in abstract algebra and can be found in \cite{cite8}.
\end{proof}
\begin{claim}
The order of $\bar{b}$ is prime and, in particular, $\ordbb = p$.
\end{claim}
\begin{proof}
Assume not and let $\ordbb=p_{1}p_{2}m$, where $p_{1}$ and $p_{2}$
are prime and $m$ is a positive integer. Then $p_{1}||G_{m}^{*}/B|$
so there is an element $\bar{g}\in G_{m}^{*}/B$ such that $\ordg=p_{1}$.
Hence $|\langle\bar{g}\rangle|=p_{1}$ and $\bar{b}\in\langle\bar{g}\rangle$, a contradiction.

Now, $\ordbb = p$ since if there is another $kb \in B$ such that $1<k<p$, then we would be able to eventually get $b\in B$, which contradicts the choice of $B$.  
\end{proof}
\begin{claim}
$G_{m}^{*}/B$ is a $p$-group.
\end{claim}
\begin{proof}
Assume $|G_{m}^{*}/B|=p^{k}p_{1}$, where $p_{1}\neq p$.
Then the claim follows by the same reasoning as the previous one.
\end{proof}
\begin{claim}
$G_{m}^{*}/B$ is cyclic.
\end{claim}
\begin{proof}
Let $x\in G_{m}^{*}/B$ be an element of maximal order, say $p^{n_{1}}$. Then the order of $\langle p^{n_{1}-1}x \rangle $ is $p$. Since $H =\langle \bar{b} \rangle$ is contained in every subgroup and its order is $p$, it follows that $H =  \langle p^{n_{1}-1}x \rangle$. Let $y\in G_{m}^{*}/B$ be such that $y\notin \langle x \rangle$ and $y$ is of minimal order, say $p^{n_{2}}$. Then there is a $0<k<p$ such that $p^{n_{2}-1}y=kp^{n_{1}-1}x$. From there we get $p^{n_{2}-1}(y-kp^{n_{1}-n_{2}}x)=0$. Now, by minimality of $n_{2}$, it must be that $y-kp^{n_{1}-n_{2}}x\in \langle x \rangle$, so $y\in \langle x \rangle$ and hence $G_{m}^{*}/B$ is cyclic.
\end{proof}

We are now going to define a partial homomorphism from $G_{m}^{*}/B$
to $\mathbb{Q}/\mathbb{Z}$. If $p=2$, send $\bar{b}$ to $\dfrac{1}{2}$ and if $p>2$ send $\bar{b}$ to
$\dfrac{p-1}{2p}$. In each case, color $\mathbb{Q}/\mathbb{Z}$ by $c$ with $2n$ or $\left\lceil \dfrac{2np}{p-1}\right\rceil$ colors the same way as it is colored in the original proof of Straus' theorem.
For all $x\in G_{m}$,
let $c(x)=c(\bar{x})$. Now, if we assume $c(x)=c(y)$, then in each case we get a contradiction. 

Thus, for every $m$, there is a $k$-coloring of $G_{m}$ with no
pairwise monochromatic solutions to $(1)$. This means that our tree is
infinite and the claim follows from Weak K\"onig's Lemma.
\end{proof}

In the following theorem, we investigate the reverse mathematical version of Theorem 3.2. The statements $\textup{(iii)}_{N, N\geq k}$ form an infinite schema of statements for every $N\geq k$, where $k$ is as stated below.

\begin{manualthm}{4.3}
The following statements are equivalent over $\textup{RCA}_0,\textup{:}$
\begin{itemize}
\item[$\textup{(i)}$] $\textup{WKL}_0$
\item[$\textup{(ii)}$] Let $n\in \mathbb{N},n\geq2$. Let $(G,+,0)$ be an abelian group and $b\in G, b\neq 0$. Then there is a $k$-coloring of $G$ with no pairwise monochromatic solutions to 
\begin{equation*}
(x_{1}-y_{1})+\dots+(x_{n}-y_{n})=b
\end{equation*}
that satisfies the following condition: there is an $m>0$ such that for every $l\in \mathbb{Z}$, all the $lmb$ are colored with the same color, where
\[
k=\begin{cases}
\hspace{4mm}2n & \text{if }2\mid \ordb \text{ or } \ordb=\infty\\
\left\lceil \dfrac{2np}{p-1}\right\rceil  & \text{if }\ordb \text{ is odd and }p\text{ is the largest prime divisor of }\ordb
\end{cases}.
\]

\item[$\textup{(iii)}_{N \geq k}$]  Let $n\in \mathbb{N},n\geq2$. Let $(G,+,0)$ be an abelian group and $b\in G, b\neq 0$. Then there is an $N$-coloring of $G$ with no pairwise monochromatic solutions to 
\begin{equation*}
(x_{1}-y_{1})+\dots+(x_{n}-y_{n})=b
\end{equation*}
that satisfies the following condition: there is an $m>0$ such that for every $l\in \mathbb{Z}$, all the $lmb$ are colored with the same color, where 
\[
k=\begin{cases}
\hspace{4mm}2n & \text{if }2\mid \ordb \text{ or } \ordb=\infty\\
\left\lceil \dfrac{2np}{p-1}\right\rceil  & \text{if }\ordb \text{ is odd and }p\text{ is the largest prime divisor of }\ordb
\end{cases}.
\]
\end{itemize}

\end{manualthm}

\begin{proof}

The fact that $\textup{(ii)}\Rightarrow\textup{(iii)}_N$ is trivial since if you have a $k$-coloring that satisfies the conditions above, that coloring can be considered an $N$-coloring for every $N\geq k$. Using Theorem 4.2 and the fact that Straus' theorem proves the additional condition in $\textup{(ii)}$, we get $\textup{(i)}\Rightarrow \textup{(ii)}$. Finally, using similar reasoning as in Theorem 4.1, $\textup{(ii)}\Rightarrow \textup{(i)}$ and $\textup{(iii)}_N\Rightarrow \textup{(i)}$  follow from the proof of Theorem 3.2 and the fact that $\textup{WKL}_0$ is equivalent to the existence of a $\textup{DNC}_k$ function for a fixed $k\in\mathbb{N},k\geq 2$.
\end{proof}  

Next, we give a reverse-mathematical version of Theorem 3.3.

\begin{manualthm}{4.4}
The following statements are equivalent over $\textup{RCA}_0\textup{:}$
\begin{itemize}
\item[$\textup{(i)}$] $\textup{WKL}_0$
\item[$\textup{(ii)}$]  Let $n\in \mathbb{N},n\geq2$. Let $(G,+,0)$ be an abelian group and $b\in G, b\neq 0$. Then there is a finite coloring of $G$ with no pairwise monochromatic solutions to 
\begin{equation*}
(x_{1}-y_{1})+\dots+(x_{n}-y_{n})=b
\end{equation*}
that satisfies the following condition: there is an $m>0$ such that for every $l\in \mathbb{Z}$, all the $lmb$ are colored with the same color.
\end{itemize}
\end{manualthm}

\begin{proof}
From Theorem 4.3, we have that $\textup{(i)}\Rightarrow\textup{(ii)}$. The other direction follows directly from the proof of Theorem 3.3 together with Lemma 2.3.
\end{proof}

As previously, the * versions of these theorems hold as well.

\vspace{1cm}

\section{The Path Ahead}

\vspace{5mm}

As mentioned before, this paper is the first paper that analyses partition regularity over algebraic structures from a computability-theoretic and reverse-mathematical point of view. That means that there are many interesting directions for future research and some pressing open questions. Here are some of them.
\subsection{Optimal bounds for the full Straus' theorem}

While we have made significant progress on the computability-theoretic and reverse-mathematical bounds of different cases of Straus' Theorem, there are some questions regarding the full theorem that are still open.

\begin{openq}
Are the PA degrees the best possible computability-theoretic bound for the full Straus' theorem? If not, what are the best possible bounds?
\end{openq}

Towards this, proposition 3.1 says that the PA degrees are an upper bound.

On the reverse mathematics side, the following question is still open.

\begin{openq}
Is the full Straus' theorem equivalent to $\textup{WKL}_0$ over $\textup{RCA}_0$?
\end{openq}

Theorem 4.2 gives us one direction of this possible equivalence.

\subsection{Reverse mathematics over $\textup{RCA}_{0}^*$}

Throughout the paper, we work over the standard weak base system of reverse mathematics $\textup{RCA}_0$. However, we can work over weaker systems that can further reveal connections and subtleties between mathematical statements. One such system is $\textup{RCA}_{0}^*$. The system $\textup{RCA}_{0}^*$ consists of the same axioms as $\textup{RCA}_0$ but with weaker induction instead of $\Sigma_0^1$-induction. In particular, we replace $\Sigma_0^1$-induction with $\Sigma_0^0$-induction, we add a symbol for exponentiation in the language of second-order arithmetic, and we add the rules for exponentiation to our first-order axioms. 

An important thing to note is that $\textup{RCA}_{0}^*$ does not prove $\Pi_0^1$-induction and this is something we relied on in Lemma 2.3, and therefore, Theorem 4.4. Thus, we have the following open questions.

\begin{openq}
Does Lemma 2.3 hold over $\textup{RCA}_{0}^*$?
\end{openq}

\begin{openq}
Does Theorem 4.4 hold over $\textup{RCA}_{0}^*$?
\end{openq}

\subsection{Computability-theoretic and reverse-mathematical analysis of the Leader-Russell theorem}

A computability-theoretic and reverse-mathematical analysis of the Leader-Russell theorem is a natural next step to make.
If Theorem 3.1 could be translated to a statement about computable
rings, then we would have a result that states that the Leader-Russell
theorem does not hold computably. However, the proofs that we have
for these theorems do not immediately generalize to commutative rings and the question of obtaining
a commutative ring from an infinitely generated abelian group seems to be a very deep algebraic question. That suggests that we might
need a completely new construction of a computable commutative ring
for which all the colorings with monochromatic solutions to a well-chosen
system of equations are noncomputable.

\begin{openq}

What are the best possible computability-theoretic bounds for the Leader-Russell theorem? What about the reverse-mathematical bounds?

\end{openq}

\subsection{Other more refined reducibilities}

The proofs of Theorem 3.2 and Theorem 3.3 heavily rely on Jockusch's theorem that says that the $\textup{DNC}_{k}$
degrees coincide with the $\textup{PA}$ degrees. These notions are equivalent
when considering Turing reducibility, which is the most widely used
measure of complexity in computability theory. There are other finer
reducibilities, however, for which the $\textup{DNC}_{k}$ degrees
are different for every $k\in\mathbb{N},k\geq 2$. One example is the Weihrauch degrees.

\begin{defn} Let $\varphi$ and $\psi$ be two arithmetic formulas in second-order arithmetic. A principle $\textup{P}$ of the form $$(\forall X)[\varphi(X)\rightarrow (\exists Y)\psi(X,Y)]$$
is called a $\Pi^1_2$-\textit{principle} or a $\Pi^1_2$-\textit{problem}. Every $X$ for which $\varphi(X)$ holds is called an \textit{instance} of $\textup{P}$ and given $X$, any $Y$ that satisfies $\psi(X,Y)$ is called a \textit{solution} to $X$.
\end{defn}

\begin{defn} Let $\textup{P}$ and $\textup{Q}$ be two $\Pi^1_2$-problems. We say that $\textup{P}$ is \textit{Weihrauch reducible} to $\textup{Q}$, denoted by $\textup{P} \leq_W \textup{Q}$, if there are Turing functionals $\Phi$ and $\Psi$ such that for every instance $X$ of $\textup{P}$, $\hat{X}=\Phi^X$ is an instance of $Q$ and for every solution $\hat{Y}$ to $\hat{X}$, $\Psi^{X \oplus \hat{Y}}$  is a solution to $X$.
\end{defn}

We point out that there are more general definitions of Weihrauch reducibility in the literature, see for example \cite{cite4}.

In \cite{cite13}, Jockusch also showed that given $k\geq2$, there is no Turing functional $\Phi$ such that $\Phi(g)\in \textup{DNC}_k$ for all $g\in\textup{DNC}_{k+1}$. If by $\textup{DNC}_k$ we denote the $\Pi^1_2$-principle that says that for every $X$, there exists a $\textup{DNC}_k$ function computable from $X$, then we can translate this into a statement about Weihrauch reducibility.

\begin{manualthm}{}[$\textup{Jockusch}_W$]

Let $k\in\mathbb{N},k\geq 2$. Then $\textup{DNC}_k\not\leq_W \textup{DNC}_{k+1}$.
\end{manualthm}

Thus, we have the following more general open question.

\begin{openq}
Will our theorems still hold under Weihrauch reducibility instead of Turing reducibility? If not, what can we say about them with respect to the Weihrauch lattice?
\end{openq}

\subsection{The case $b=0$}

Our current research focuses on partition regularity
for inhomogeneous systems ($Ax=b,b\neq0)$, inspired by the recent
result by Leader and Russell. But the homogeneous case $b=0$ is interesting
in its own right. Rado showed that a homogeneous system is partition
regular over $\mathbb{Z}$ if and only if it satisfies a computable
condition called ``the columns condition''. The problem of characterizing all such partition regular systems over different algebraic structures has been extensively studied. In \cite{cite18} and \cite{cite3}, the authors consider commutative rings, in \cite{cite6} and \cite{cite7} abelian groups and in \cite{cite2} vector spaces over finite fields. Byszewski and Krawczyk
in \cite{cite5} showed that the same characterization using the columns condition also holds if we consider
partition regularity over integral domains. Therefore, a similar analysis of
the case $b=0$ over different algebraic structures is a natural possible continuation of our work.

\subsection{The nonlinear case}

Many researchers working in combinatorics have been interested in
the partition regularity of specific nonlinear equations, see for
example, \cite{cite1}, \cite{cite10} and \cite{cite16}. While the characterization
of all partition regular nonlinear equations seems like a daunting
task, similar computability-theoretic and reverse-mathematical analysis
as in our work can be done of the already existing results. It would
be interesting to see whether the computability-theoretic complexity
changes significantly depending on the combinatorial complexity of
the equations.

\end{document}